\documentclass[12pt, reqno]{amsart}
\usepackage{amsmath, amsthm, amscd, amsfonts, amssymb, graphicx, color}
\usepackage[bookmarksnumbered, colorlinks, plainpages]{hyperref}
\hypersetup{colorlinks=true,linkcolor=red, anchorcolor=green, citecolor=cyan, urlcolor=red, filecolor=magenta, pdftoolbar=true}

\newtheorem{theorem}{Theorem}[section]
\newtheorem{lemma}[theorem]{Lemma}

\newtheorem{cor}[theorem]{Corollary}
\theoremstyle{definition}

\theoremstyle{remark}
\newtheorem{remark}[theorem]{\bf{Remark}}
\numberwithin{equation}{section}
\begin{document}

\title [ Improvement of $A$-numerical radius inequalities ]{ Improvement of $A$-numerical radius inequalities of semi-Hilbertian space operators }

\author[P. Bhunia, R.K. Nayak and K. Paul]{Pintu Bhunia, Raj Kumar Nayak and Kallol Paul}

\address{(Bhunia) Department of Mathematics, Jadavpur University, Kolkata 700032, India}
\email{pintubhunia5206@gmail.com}

\address{(Nayak) Department of Mathematics, Jadavpur University, Kolkata 700032, India}
\email{rajkumarju51@gmail.com}

\address{(Paul) Department of Mathematics, Jadavpur University, Kolkata 700032, India}
\email{kalloldada@gmail.com;kallol.paul@jadavpuruniversity.in}

%\thanks will become a 1st page footnote.
\noindent \thanks{First and second authors would like to thank UGC, Govt. of India for the financial support in the form of SRF. Prof. Kallol Paul would like to thank RUSA 2.0, Jadavpur University for the partial support.}
\thanks{}
\thanks{}
%    Information for second author

%    General info


%\date{}
\maketitle

\begin{abstract}
Let $\mathcal{H}$ be a complex Hilbert space and let $A$ be a positive operator on $\mathcal{H}$. We obtain new bounds for the $A$-numerical radius of operators in semi-Hilbertian space $\mathcal{B}_A(\mathcal{H})$ that generalize and improve on the existing ones.   
Further, we estimate bounds for the $B$-operator seminorm and $B$-numerical radius of $2\times 2$ operator matrices, where $B=\mbox{diag}(A,A)$. The bounds obtained here improve on the existing bounds.
\end{abstract}

\section{Introduction}

\noindent The purpose of the present article is to study the numerical radius inequalities of semi-Hilbertian space operators and operator matrices, which  generalize the classical numerical radius inequalities of complex Hilbert space operators and operator matrices. The motivation comes from the recent papers \cite{Kittaneh:3, BP}. Let us first introduce the following notations and terminologies.\\ 
	
\noindent

 Let $\mathcal{B}(\mathcal{H})$ denote the $\mathbb{C}^*$-algebra of all bounded linear operators on a complex Hilbert space $\mathcal{H}$ with inner product $ \langle .,.\rangle $ and the corresponding norm $\|.\|.$ Let $T\in \mathcal{B}(\mathcal{H}).$ As usual  the Range of $T$ and the Kernel of $T $ are denoted by $R(T)$ and $N(T)$, respectively. By $\overline{\mathcal{R}(T)}$ we denote the norm closure of $\mathcal{R}(T)$. Let $T^*$ be the adjoint of $T$.
The letters $I$ and $O$ are reserved  for the identity operator and the zero operator on $\mathcal{H}$, respectively. Throughout the article, $A \in \mathcal{B}(\mathcal{H})$ is a positive operator and $B = \left(\begin{array}{cc}
	A&O \\
	O&A
	\end{array}\right)$. Clearly, $A$ induces a positive semidefinite sesquilinear form $\langle . , . \rangle_A: \mathcal{H} \times \mathcal{H} \rightarrow \mathbb{C}$, defined by $ \langle x, y \rangle_A = \langle Ax, y \rangle$ for all $x,y\in \mathcal{H}.$ This sesquilinear form induces a seminorm $\|.\|_A: \mathcal{H}\rightarrow \mathbb{R}^+$, defined by $\|x\|_A = \sqrt{\langle x, x \rangle_A}$  for all $ x \in \mathcal{H}.$
Clearly, $\|.\|_A$ is a norm if and only if $A$ is injective and $(\mathcal{H}, \|.\|_A)$ is complete if and only if $R(A)$ is closed in $ \mathcal{H}.$  An operator $R \in \mathcal{B}(\mathcal{H})$ is called an A-adjoint of $T$ if  $\langle Tx, y \rangle_A = \langle x, Ry \rangle_A$ for all $x,y\in \mathcal{H}$. The existence of an A-adjoint of $T$ is not guaranteed. Let $ \mathcal{B}_A(\mathcal{H})$ denote the collection of all operators in $\mathcal{B}(\mathcal{H})$, which admit A-adjoints.  By Douglas theorem \cite{Douglas}, it follows that 
$$ \mathcal{B}_A(\mathcal{H}) = \{ T \in \mathcal{B}(\mathcal{H}) : R(T^*A) \subseteq R(A) \}.$$ 
If $T\in \mathcal{B}_A(\mathcal{H})$ then the operator equation $AX=T^*A$ has a unique solution, denoted by $T^{\sharp_A}$, satisfying $\mathcal{R}(T^{\sharp_A})\subseteq \overline{\mathcal{R}(A)}$. Note that $T^{\sharp_A}=A^\dag T^*A$, where $A^\dag$ is the Moore-Penrose inverse of $A$. Clearly, $A^{\sharp_A} = A. $
For $T \in \mathcal{B}_A(\mathcal{H})  $,  we have $AT^{\sharp_A}= T^*A$ and $ N(T^{\sharp_A}) = N(T^*A).$ Note that, if $T \in  \mathcal{B}_A(\mathcal{H})$ then $T^{\sharp_A} \in  \mathcal{B}_A(\mathcal{H})$  and $(T^{\sharp_A})^{\sharp_A} = PTP$, where $P$ is the orthogonal projection onto $\overline{R(A)}$.
For further study on the A-adjoint of an operator, we refer the interested readers to \cite{Arias1}.  Let us now define the $A$-operator seminorms on $\mathcal{B}_A(\mathcal{H})$. Let $T \in \mathcal{B}_{A}(\mathcal{H})$. The $A$-operator seminorm of $T$, denoted by $\|T\|_A$, is defined as
	\begin{eqnarray*}
		 \|T\|_A =  \sup \{\|Tx\|_A: x\in \mathcal{H}, \|x\|_A= 1\}.
	\end{eqnarray*}
Clearly, $\|TT^{\sharp_A}\|_A = \|T^{\sharp_A}T\|_A = \|T^{\sharp_A}\|^2_A = \|T\|^2_A. $
The A-minimum norm of $T$, denoted by $m_A(T)$, is defined as 
\[m_A(T)=  \inf \{\|Tx\|_A : x\in \mathcal{H}, \|x\|_A= 1\}.\]
The $A$-numerical range and the $A$-numerical radius of $T$, denoted by $W_A(T)$ and $w_A(T)$, respectively, are defined as
\begin{eqnarray*}
W_A(T) &=& \{\langle Tx, x \rangle_A: x\in \mathcal{H}, \|x\|_A= 1 \},\\
w_A(T) &=& \sup \{|\langle Tx, x \rangle_A|: x\in \mathcal{H}, \|x\|_A= 1 \}.
\end{eqnarray*}
For $T \in \mathcal{B}_{A}(\mathcal{H})$, we also have $$\|T\|_A = \sup \{|\langle Tx, y \rangle_A|: x,y\in \mathcal{H}, \|x\|_A=\|y\|_A= 1 \}.$$
In particular, if we consider $A=I$ in the definitions of $A$-operator seminorm, A-minimum norm and A-numerical radius then we have the classical operator norm, minimum norm and  numerical radius, respectively, i.e., $\|T\|_A=\|T\|$, $m_A(T)=m(T)$ and $w_A(T)=w(T)$.
 It is well-known that $w_A(.)$ and $\|.\|_A$  are equivalent seminorms on $\mathcal{B}_{A}(\mathcal{H})$, satisfying the following inequality
\[ \frac{1}{2} \|T\|_A \leq w_A(T) \leq \|T\|_A.\]
Recently many eminent mathematicians have studied  $A$-numerical radius inequalities, we refer the interested readers to \cite{ RajAOT, Pintu1, Moslehian, Xu} and references therein.\\
\noindent This article is seperated into three sections, including the introductory one. In the second section, we develop inequalities for the $A$-numerical radius of operators in $\mathcal{B}_{A}(\mathcal{H})$. The inequalities obtained here generalize and improve on the inequalities in \cite{Zamani}. In particular, we show that if $T\in \mathcal{B}_{A}(\mathcal{H})$ then the following inequalities hold
\begin{eqnarray*}
w^2_A(T) \leq \min_{0\leq \alpha\leq 1} \left \|\alpha T^{\sharp_A}T + (1- \alpha)TT^{\sharp_A}\right \|_A,
\end{eqnarray*}
\[ w^2_A(T) = \min_{0\leq \alpha \leq 1} \left \{\frac{\alpha}{2} w_A(T^2) +\left\| \frac{\alpha}{4}   TT^{\sharp_A}+\left(1-\frac{3\alpha}{4}\right)  T^{\sharp_A}T  \right\|\right \}\] 
and
\[w^2_A(T) = \min_{0\leq \alpha \leq 1} \left \{\frac{\alpha}{2} w_A(T^2) +\left\|  \left(1-\frac{3\alpha}{4}\right)   TT^{\sharp_A} +\frac{\alpha}{4}  T^{\sharp_A} T \right\|\right \}.\]
In third section we study the inequalities on operator matrices. We obtain bounds for the $B$-numerical radius of $2\times 2$ operator matrices of the form  $\left(\begin{array}{cc}
X&Y \\
Z&W
\end{array}\right)$, where $X, Y,Z,W \in \mathcal{B}_A(\mathcal{H})$. Also, we obtain an upper bound for the $B$-operator seminorm of $\left(\begin{array}{cc}
X&Y \\
Z&W
\end{array}\right)$. We show that the bounds obtained here improve on the existing ones.

\section{\textbf{ Inequalities of operators }}

We begin with the following theorem that gives an upper bound for the A-numerical radius of  bounded linear operators on $\mathcal{H}$ that admits $A$-adjoint.

\begin{theorem} \label{Theorem:1}
Let $T \in  \mathcal{B}_A(\mathcal{H}) $. Then 
\[w^2_A(T) \leq \min_{0\leq \alpha\leq 1} \left \|\alpha T^{\sharp_A}T + (1- \alpha)TT^{\sharp_A}\right \|_A.\]

\end{theorem}

\begin{proof}
	Let $x \in \mathcal{H}$ with $\|x\|_A = 1$. Then for all $ \alpha \in [0,1]$  we get,
	\begin{eqnarray*}
		|\langle Tx, x \rangle_A| & =& \alpha 	|\langle Tx, x \rangle_A| + (1- \alpha) 	|\langle x, T^{\sharp_A}x \rangle_A|\\
	\Rightarrow	|\langle Tx, x \rangle_A| &\leq& \alpha \|Tx\|_A + (1- \alpha) \|T^{\sharp_A}x\|_A\\
		\Rightarrow |\langle Tx, x \rangle_A|^2 &\leq&  \alpha \|Tx\|_A^2 + (1- \alpha) \|T^{\sharp_A}x\|_A^2, ~~~~\mbox{by convexity of $t^2$}\\
	\Rightarrow |\langle Tx, x \rangle_A|^2	&\leq & \alpha \langle Tx, Tx \rangle_A + (1- \alpha) \langle T^{\sharp_A}x, T^{\sharp_A}x \rangle_A\\
	\Rightarrow |\langle Tx, x \rangle_A|^2	& \leq & \alpha \langle T^{\sharp_A}Tx, x \rangle_A + (1- \alpha) \langle TT^{\sharp_A}x, x \rangle_A\\
	\Rightarrow |\langle Tx, x \rangle_A|^2	& \leq &  \big\langle \{\alpha T^{\sharp_A}T +  (1- \alpha)TT^{\sharp_A}\}x, x \big\rangle_A\\
	\Rightarrow |\langle Tx, x \rangle_A|^2	&\leq & \big\|\alpha T^{\sharp_A}T +  (1- \alpha)TT^{\sharp_A}\big\|_A.
	\end{eqnarray*}
	Taking supremum over $\|x\|_A=1,$ we get 
	\[w^2_A(T) \leq  \left \|\alpha T^{\sharp_A}T + (1- \alpha)TT^{\sharp_A}\right \|_A,~~\forall \alpha \in [0,1].\]
	Taking minimum over all $ \alpha \in [0,1]$ we get the desired inequality.
\end{proof}

\begin{remark}\label{rem1}
 In  \cite[Th. 2.10]{Zamani},  Zamani proved that 
\begin{eqnarray}\label{1a}
w^2_A(T) \leq \frac{1}{2} \left \| T^{\sharp_A}T + TT^{\sharp_A} \right \|_A,
\end{eqnarray}
which follows clearly from Theorem \ref{Theorem:1}, in fact, we have

\[w^2_A(T) \leq \min_{0\leq \alpha\leq 1} \left \|\alpha T^{\sharp_A}T + (1- \alpha)TT^{\sharp_A}\right \|_A\leq \frac{1}{2} \left \| T^{\sharp_A}T + TT^{\sharp_A} \right \|_A.\]

Considering 
 $$T =  \left(\begin{array}{ccc}
	0&2&0 \\
	0&0&1\\
	0&0&0
	\end{array}\right)~~ \mbox{and}~~ A =  \left(\begin{array}{ccc}
	1&0&0 \\
	0&2&0\\
	0&0&3
	\end{array}\right)$$  we get, 
\begin{eqnarray*}
\min_{0\leq \alpha\leq 1} \left \|\alpha T^{\sharp_A}T + (1- \alpha)TT^{\sharp_A}\right \|_A &=& 
\frac{6}{5}\\
\mbox{and}~~\frac{1}{2} \left \| T^{\sharp_A}T + TT^{\sharp_A} \right \|_A & = & \frac{4}{3}.
\end{eqnarray*}
Thus, we observe that  Theorem \ref{Theorem:1} is a non-trivial improvement of  \cite[Th. 2.10]{Zamani}.
\end{remark}

To develop the next inequality we need the following lemma.

\begin{lemma} \cite{Aniket} \label{Lemma:3} 
	Let $x, y, e \in \mathcal{H}$ with $\|e\|_A = 1.$ Then
	\begin{eqnarray*}
		|\langle a, e \rangle_A \langle e, b \rangle_A| \leq \frac{1}{2}\big( |\langle a, b \rangle_A| + \|a\|_A\|b\|_A\big).
	\end{eqnarray*}
\end{lemma}

We now obtain the following inequality for the $A$-numerical radius of operators in $\mathcal{B}_A(\mathcal{H})$.

	\begin{theorem} \label{Theorem:4}
		Let $T \in \mathcal{B}_A(\mathcal{H})$ and let $r\geq 1$. Then
	$$ w^{2r}_A(T) \leq \frac{1}{2} w^r_A(T^2) + \frac{1}{2^{r+1}} \left\|TT^{\sharp_A} + T^{\sharp_A}T\right\|^r_A. $$
	\end{theorem}

\begin{proof}
Let $x\in \mathcal{H}$ with $\|x\|_A=1$. Taking $a=Tx$, $ b= T^{\sharp_A}x$ and $ e=x$ in Lemma \ref{Lemma:3}, we get
\begin{eqnarray*}
| \langle Tx, x \rangle_A|^2 &\leq& \frac{1}{2}\left(| \langle T^2x, x \rangle_A|+\|Tx\|_A ~~\|T^{\sharp_A}x\|_A \right)\\
\Rightarrow | \langle Tx, x \rangle_A|&\leq& \frac{1}{2}| \langle T^2x, x \rangle_A|+\frac{1}{4}\left (\|Tx\|^2_A +\|T^{\sharp_A}x\|^2_A \right), ~~\mbox{by AM-GM inequality}\\
\Rightarrow | \langle Tx, x \rangle_A|&\leq & \frac{1}{2}| \langle T^2x, x \rangle_A|+\frac{1}{4} \left \langle \left (TT^{\sharp_A} + T^{\sharp_A}T\right)x, x \right\rangle_A \\
\Rightarrow | \langle Tx, x \rangle_A|&\leq& \frac{1}{2} w_A(T^2) + \frac{1}{4} \|TT^{\sharp_A} + T^{\sharp_A}T\|_A\\
\Rightarrow | \langle Tx, x \rangle|_A^{2r} &\leq&  \frac{1}{2} w^r_A(T^2) + \frac{1}{2^{r+1}} \|TT^{\sharp_A} + T^{\sharp_A}T\|^r_A, ~~\mbox{by convexity of}~~ t^r.
 \end{eqnarray*}
Taking supremum over $\|x\|_A =1$, we get the required inequality.
	\end{proof}

\begin{remark}

\noindent 1. In particular, if we consider $r=1$ in Theorem \ref{Theorem:4} then we get the inequality in \cite[Th. 2.11]{Zamani}.

\noindent 2. In \cite[Th. $2.4$]{SMY}, Sattari et. al. proved that the following numerical radius inequality
	$$ w^{2r}(T) \leq \frac{1}{2} \left(w(T^2) + \|T\|^{2r}\right),$$ for $T \in \mathcal{B}(\mathcal{H})$ and $r\geq 1$.
	Clarly, if we consider $A=I$ then Theorem \ref{Theorem:4} gives better bound than that in \cite[Th. $2.4$]{SMY}.\\
\noindent 	3. It is pertinent to mention here that there was a mathematical mistake in the proof of  a similar inequality developed in \cite[Th. 2.16]{Rajejla}, the mistake was in the consideration of  $TT^{\sharp_A} $ and $  T^{\sharp_A}T$ as positive operators, which is not necessarily true.
\end{remark}

In our next theorem we obtain upper bounds for the A-numerical radius of operators in $\mathcal{B}_A(\mathcal{H})$ which generalize the inequality in \cite[Th. 2.11]{Zamani}.

\begin{theorem} \label{Theorem:3}
	Let $T \in \mathcal{B}_A(\mathcal{H}).$ Then for all $\alpha \in[0,1]$,
	\begin{eqnarray}\label{3a}
w_A^2(T ) \leq \frac{\alpha}{2} w_A(T^2) + \bigg\|\frac{\alpha}{4}TT^{\sharp_A} + \bigg(1- \frac{3\alpha}{4} \bigg) T^{\sharp_A}T\bigg\|_A
\end{eqnarray}
and
\begin{eqnarray}\label{3b}
w_A^2(T ) \leq \frac{\alpha}{2} w_A(T^2) + \bigg\| \bigg(1- \frac{3\alpha}{4} \bigg)TT^{\sharp_A} + \frac{\alpha}{4}T^{\sharp_A}T\bigg\|_A.
	\end{eqnarray}
\end{theorem}

\begin{proof}
Let $x \in \mathcal{H}$ with $\|x\|_A = 1$.
Considering $a=Tx, b= T^{\sharp_A}x, e=x$ in Lemma \ref{Lemma:3} and then by using AM-GM inequality, we get
\begin{eqnarray*} 
|\langle Tx, x \rangle_A|^2 \leq \frac{1}{2}|\langle T^2x, x \rangle_A| + \frac{1}{4} \langle (TT^{\sharp_A} + T^{\sharp_A}T)x, x \rangle_A.
\end{eqnarray*}
Now,
\begin{eqnarray*}
	|\langle Tx, x \rangle_A| &=& \alpha 	|\langle Tx, x \rangle_A| + (1- \alpha) 	|\langle Tx, x \rangle_A|\\
	\Rightarrow |\langle Tx, x \rangle_A| &\leq&  \alpha 	|\langle Tx, x \rangle_A| + (1-\alpha) \|Tx\|_A ,~~~\mbox{by Cauchy-Schwarz inequality}\\
	\Rightarrow |\langle Tx, x \rangle_A|^2 &\leq & \alpha 	|\langle Tx, x \rangle_A|^2 + (1-\alpha) \|Tx\|_A^2, ~~\mbox{by convexity of}~~ t^2\\
	\Rightarrow |\langle Tx, x \rangle_A|^2 &\leq& \frac{\alpha}{2}|\langle T^2x, x \rangle_A| + \frac{\alpha}{4} \langle (TT^{\sharp_A} + T^{\sharp_A}T)x, x \rangle_A + (1- \alpha) \langle T^{\sharp_A}Tx, x \rangle_A\\
	\Rightarrow |\langle Tx, x \rangle_A|^2 &\leq&  \frac{\alpha}{2}|\langle T^2x, x \rangle_A| + \big\langle \left\{\frac{\alpha}{4}TT^{\sharp_A} +\big(1- \frac{3\alpha}{4} \big)T^{\sharp_A}T \right \}x, x \big\rangle_A\\
	\Rightarrow |\langle Tx, x \rangle_A|^2 &\leq&  \frac{\alpha}{2} w_A(T^2) + \bigg\|\frac{\alpha}{4}TT^{\sharp_A} + \bigg( 1- \frac{3\alpha}{4} \bigg)T^{\sharp_A}T \bigg\|_A.
	\end{eqnarray*}
Taking supremum over $\|x\|_A = 1$, we get  the inequality (\ref{3a}). Similarly as in (\ref{3a}),  we can prove the inequality (\ref{3b}).
\end{proof}

As a consequence of Theorem \ref{Theorem:3}, we get the following corollary.

\begin{cor} \label{Cor3}
	Let $T \in \mathcal{B}_A(\mathcal{H}).$ Then 
	\begin{eqnarray}\label{bound2comb}
	w^{2}(T)\leq \min \left \{\beta_1, \beta_2  \right \}, 
	\end{eqnarray}
	\mbox{where}
	\[ \beta_1 = \min_{0\leq \alpha \leq 1} \left \{\frac{\alpha}{2} w_A(T^2) +\left\| \frac{\alpha}{4}   TT^{\sharp_A}+\left(1-\frac{3\alpha}{4}\right)  T^{\sharp_A}T  \right\|\right \}\] and
	\[\beta_2 = \min_{0\leq \alpha \leq 1} \left \{\frac{\alpha}{2} w_A(T^2) +\left\|  \left(1-\frac{3\alpha}{4}\right)   TT^{\sharp_A} +\frac{\alpha}{4}  T^{\sharp_A} T \right\|\right \}.\]
\end{cor}

\begin{remark}
In particular, if we consider $\alpha=1$ in Corollary \ref{Cor3}, then we get the inequality in \cite[Th. 2.11]{Zamani}, i.e.,
\begin{eqnarray}\label{1b}
 w_A^2(T ) \leq \frac{1}{2} w_A(T^2) + \frac{1}{4} \bigg\|TT^{\sharp_A} +  T^{\sharp_A}T\bigg\|_A.
\end{eqnarray}
Now we consider an example. Let $T$ and $A$ be the same as described in Remark \ref{rem1}. Then by elementary calculations we have, $w_A(T^2)=\frac{1}{\sqrt{3}}$. Also,  
$$\beta_1= \min_{0\leq \alpha\leq 1}\max\left\{  \frac{\alpha}{2}+\frac{\alpha}{2\sqrt{3}},~~ 2- \frac{4\alpha}{3}+\frac{\alpha}{2\sqrt{3}}, ~~\frac{2}{3} \left( 1- \frac{3\alpha}{4}\right)+\frac{\alpha}{2\sqrt{3}}    \right\}= \frac{2}{3}+\frac{1}{2\sqrt{3}}$$ and

$$\beta_2= \min_{0\leq \alpha\leq 1}\max\left\{  2-\frac{3\alpha}{2}+\frac{\alpha}{2\sqrt{3}},~~  \frac{2}{3}+\frac{\alpha}{2\sqrt{3}},~~ \frac{\alpha}{6} +\frac{\alpha}{2\sqrt{3}}    \right\}= \frac{2}{3}+\frac{4}{9\sqrt{3}}.$$
Therefore,
\[\min \left \{\beta_1, \beta_2  \right \}=\frac{2}{3}+\frac{4}{9\sqrt{3}}< \frac{2}{3}+\frac{1}{2\sqrt{3}}= \frac{1}{2} w_A(T^2) + \frac{1}{4} \bigg\|TT^{\sharp_A} +  T^{\sharp_A}T\bigg\|_A. \]
Thus, we conclude that the inequality (\ref{bound2comb}) is a non-trivial improvement of (\ref{1b}).
\end{remark}

Finally, we obtain an inequality that involves A-operator seminorm and A-minimum norm.

\begin{theorem} \label{Prop1}
	Let $T\in \mathcal{B}_A(\mathcal{H})$. Then we have,
	\[\|T\|_A^2+\max \left \{m_A^2(T), m_A^2(T^{\sharp_A})  \right \}\leq \|T^{\sharp_A}T+TT^{\sharp_A}\|_A.\]
\end{theorem}

\begin{proof}
	Let $ x \in \mathcal{H}$ with $  \|x\|_A=1 $. Then by Cauchy-Schwarz inequality, we get
	 \begin{eqnarray*}
	\|Tx\|_A^2+\|T^{\sharp_A}x\|_A^2 &=&\langle (T^{\sharp_A}T+TT^{\sharp_A})x,x\rangle_A \leq \|T^{\sharp_A}T+TT^{\sharp_A}\|_A.
	\end{eqnarray*}
	Therefore,	$\|Tx\|_A^2+ m_A^2(T^{\sharp_A}) \leq \|T^{\sharp_A}T+TT^{\sharp_A}\|_A.$
Taking supremum over $\|x\|_A = 1$, we get,
\begin{eqnarray}\label{4a}
\|T\|_A^2 + m_A^2(T^{\sharp_A}) \leq \|T^{\sharp_A}T+TT^{\sharp_A}\|_A. 
\end{eqnarray}
Similarly, $ m_A^2(T)+\|T^{\sharp_A}x\|_A^2 \leq \|T^{\sharp_A}T+TT^{\sharp_A}\|_A.$
Taking supremum over $\|x\|_A = 1$, we get,
\begin{eqnarray}\label{4b}
\|T\|_A^2 + m_A^2(T) \leq \|T^{\sharp_A}T+TT^{\sharp_A}\|_A. 
\end{eqnarray}
Combining (\ref{4a}) and (\ref{4b}), we get the desired inequality.
\end{proof}

\begin{remark}
The inequality in \cite[Th. 1]{Kias} follows from Theorem \ref{Prop1} and by using the fact (see in \cite{Zamani}) that 
	 $\|T^{\sharp_A}T+TT^{\sharp_A}\|_A \leq 4 w_A^2(T).$
\end{remark}

\section{\textbf{ Inequalities of operator matrices}}

We begin this section with the following known results, the proof of which can be found in \cite{Rajejla, Pintu2, Kias2}.

\begin{lemma} \label{Lemma:7}
	Let $X, Y, Z, W \in \mathcal{B}_A(\mathcal{H}) $. Then the following results hold:

	\begin{eqnarray*}
		&&(i) ~~ w_B \left(\begin{array}{cc}
			X&O \\
			O&Y
		\end{array}\right)=\max \left \{ w_A(X), w_A(Y)\right \}.\\
		&&(ii)~~ w_B \left(\begin{array}{cc}
			O&X \\
			Y&O
		\end{array}\right) =  w_B\left(\begin{array}{cc}
			O&Y \\
			X&O
		\end{array}\right). \\
	&&(iii)~~   
	w_B\left(\begin{array}{cc}
		O&X \\
		e^{i\theta}Y&O
	\end{array}\right) = w_B\left(\begin{array}{cc}
		O&X \\
	  Y&O
	\end{array}\right),~~\textit{For any}~~ \theta \in \mathbb{R}.\\
	&&(iv)~~
	w_B \left(\begin{array}{cc}
		X&Y \\
		Y&X
	\end{array}\right)=\max \left \{ w_A(X+Y), w_A(X-Y)\right \}.\\ 
	&&~~ \textit{In particular}, ~~w_B \left(\begin{array}{cc}
		O&Y \\
		Y&O
	\end{array}\right)=w_A(Y).	\\
  %&&(v)~~ 
  %w_B \left(\begin{array}{cc}
  %	O&X \\
  %	Y&O
  %\end{array}\right) = \frac{1}{2}\sup_{\theta \in \mathbb{R}} \|e^{i\theta}X + e^{-i\theta}Y^{\sharp_A}\|_A.\\
&&(v)~~~
\bigg\| \left(\begin{array}{cc}
	X&O \\
	O&Y
\end{array}\right)\bigg\|_B = \bigg\| \left(\begin{array}{cc}
O&X \\
Y&O
\end{array}\right)\bigg\|_B = \max\big\{\|X\|_A, \|Y\|_A\big\}.\\
&&(vi) ~~~
 \left(\begin{array}{cc}
	X&Y \\
	Z&W
\end{array}\right)^{\sharp_B} =  \left(\begin{array}{cc}
X^{\sharp_A}&Z^{\sharp_A} \\
Y^{\sharp_A}&W^{\sharp_A}
\end{array}\right).
	\end{eqnarray*}
	\end{lemma}

First we obtain an upper bound for the $B$-operator seminorm of $2\times 2$ operator matrices of the form $\left(\begin{array}{cc}
X&Y \\
Z&W
\end{array}\right)$, where $X, Y,Z,W \in \mathcal{B}_A(\mathcal{H})$.

\begin{theorem} \label{Theorem:13}
	Let $X, Y, Z, W \in \mathcal{B}_A(\mathcal{H}).$ Then
	\begin{eqnarray*}
	 \bigg\| \left(\begin{array}{cc}
	X&Y \\
	Z&W
	\end{array}\right)\bigg\|^2_B &\leq& \max\{\|X\|_A^2, \|W\|_A^2\}+ \max\{\|X\|_A, \|W\|_A\}\max\{\|Y\|_A, \|Z\|_A \} \\
	&&+ \max\{\|Y\|_A^2, \|Z\|_A^2\}+ w_B\left(\begin{array}{cc}
	O&Z^{\sharp_A}W \\
	Y^{\sharp_A}X&O
	\end{array}\right).
\end{eqnarray*} 
\end{theorem}

\begin{proof}
Let $x, y$ be two B-unit vectors in $ \mathcal{H} \oplus \mathcal{H}.$  Then, we get 
\begin{eqnarray*}
 && \bigg|\bigg\langle \left(\begin{array}{cc}
		X&Y \\
		Z&W
	\end{array}\right)x, y \bigg\rangle_B\bigg|^2 \\
	&=& \bigg|\bigg\langle \bigg[\left(\begin{array}{cc}
	X&O \\
	O&W
\end{array}\right) + \left(\begin{array}{cc}
O&Y \\
Z&O
\end{array}\right)\bigg] x, y \bigg\rangle_B\bigg|^2\\
&=& \bigg|\bigg\langle \left(\begin{array}{cc}
	X&O \\
	O&W
\end{array}\right)x, y \bigg\rangle_B + \bigg\langle \left(\begin{array}{cc}
O&Y \\
Z&O
\end{array}\right)x, y \bigg\rangle_B  \bigg|^2 \\
&\leq& \bigg|\bigg\langle \left(\begin{array}{cc}
	X&O \\
	O&W
\end{array}\right)x, y \bigg\rangle_B\bigg|^2 +  \bigg|\bigg\langle \left(\begin{array}{cc}
O&Y \\
Z&O
\end{array}\right)x, y \bigg\rangle_B\bigg|^2 \\
&&+ 2  \bigg|\bigg\langle \left(\begin{array}{cc}
X&O \\
O&W
\end{array}\right)x, y \bigg\rangle_B\bigg| \bigg|\bigg\langle \left(\begin{array}{cc}
O&Y \\
Z&O
\end{array}\right)x, y \bigg\rangle_B\bigg|\\
&=& \bigg|\bigg\langle \left(\begin{array}{cc}
	X&O \\
	O&W
\end{array}\right)x, y \bigg\rangle_B\bigg|^2 +  \bigg|\bigg\langle \left(\begin{array}{cc}
	O&Y \\
	Z&O
\end{array}\right)x, y \bigg\rangle_B\bigg|^2 \\
&&+ 2  \bigg|\bigg\langle \left(\begin{array}{cc}
	X&O \\
	O&W
\end{array}\right)x, y \bigg\rangle_B \bigg\langle y, \left(\begin{array}{cc}
O&Y \\
Z&O
\end{array}\right)x \bigg\rangle_B\bigg| \\
&\leq&  \bigg|\bigg\langle \left(\begin{array}{cc}
	X&O \\
	O&W
\end{array}\right)x, y \bigg\rangle_B\bigg|^2 +  \bigg|\bigg\langle \left(\begin{array}{cc}
	O&Y \\
	Z&O
\end{array}\right)x, y \bigg\rangle_B\bigg|^2 \\
&&+ \bigg\|\left(\begin{array}{cc}
	X&O \\
	O&W
\end{array}\right)x\bigg\|_B \bigg\|\left(\begin{array}{cc}
O&Y \\
Z&O
\end{array}\right)x\bigg\|_B  + \bigg| \bigg\langle \left(\begin{array}{cc}
X&O \\
O&W
\end{array}\right)x, \left(\begin{array}{cc}
O&Y \\
Z&O
\end{array}\right )x \bigg\rangle_B\bigg|,\\
&& ~~\,\,\,\,\,\,\,\,\,\,\,\,\,\,\,\,\,\,\,\,\,\,\,\,\,\,\,\,\,\,\,\,\,\,\,\,\,\,\,\,\,\,\,\,\,\,\,\,\,\,\,\,\,\,\,\,\,\,\,\,\,\,\,\,\,\,\,\,\,\,\,\,\,\,\,\,\,\,\,\,\,\,\,\,\,\,\,\,\,\,\,\,\,\,\,\,\,\,\,\,\,\,\,\,\,\,\,\,\,\,\,\,\,\,\,\,\,\,\,\,\,\,\,\,\,\,\mbox{by Lemma \ref{Lemma:3}}\\
&=&  \bigg|\bigg\langle \left(\begin{array}{cc}
	X&O \\
	O&W
\end{array}\right)x, y \bigg\rangle_B\bigg|^2 +  \bigg|\bigg\langle \left(\begin{array}{cc}
	O&Y \\
	Z&O
\end{array}\right)x, y \bigg\rangle_B\bigg|^2 \\
&&+ \bigg\|\left(\begin{array}{cc}
	X&O \\
	O&W
\end{array}\right)x\bigg\|_B \bigg\|\left(\begin{array}{cc}
	O&Y \\
	Z&O
\end{array}\right)x\bigg\|_B  + \bigg| \bigg\langle \left(\begin{array}{cc}
	O&Z^{\sharp_A}W \\
	Y^{\sharp_A}X&O
\end{array}\right)x, x \bigg\rangle_B\bigg|.
\end{eqnarray*}
Taking supremum over $x, y$ with $\|x\|_B = \|y\|_B = 1$ and using Lemma \ref{Lemma:7}, we get the required inequality.
\end{proof}

\begin{remark}
We would like to note that the inequality in  \cite[Th. 2.1]{Kittaneh:3} follows from Theorem \ref{Theorem:13} by considering $A = I.$
\end{remark}

Next we obtain an upper bound for the $B$-numerical radius of $\left(\begin{array}{cc}
X&Y \\
Z&W
\end{array}\right)$, where $X, Y,Z,W \in \mathcal{B}_A(\mathcal{H})$.

\begin{theorem} \label{Theorem:14}
	Let $X, Y, Z, W \in \mathcal{B}_A(\mathcal{H}).$ Then
	\begin{eqnarray*}
		w_B^2\left(\begin{array}{cc}
			X&Y \\
			Z&W
		\end{array}\right) &\leq& \max\{w_A^2(X), w_A^2(W)\}+ w_B^2\left(\begin{array}{cc}
		O&Y \\
		Z&O
	\end{array}\right) +  w_B\left(\begin{array}{cc}
	O&Z^{\sharp_A}W \\
	Y^{\sharp_A}X&O
\end{array}\right) \\
&&+ \frac{1}{2} \max \big\{ \|X^{\sharp_A}X + Z^{\sharp_A}Z\|_A, \|W^{\sharp_A}W + Y^{\sharp_A}Y\|_A\big\}.
\end{eqnarray*} 
\end{theorem}

\begin{proof}
Let $x$ be a B-unit vector in $\mathcal{H} \oplus \mathcal{H}$. Then, we get
\begin{eqnarray*}
&& \bigg|\bigg\langle \left(\begin{array}{cc}
	X&Y \\
	Z&W
\end{array}\right)x, x \bigg\rangle_B\bigg|^2 \\
&=& \bigg|\bigg\langle \bigg[\left(\begin{array}{cc}
	X&O \\
	O&W
\end{array}\right) + \left(\begin{array}{cc}
	O&Y \\
	Z&O
\end{array}\right)\bigg] x, x \bigg\rangle_B\bigg|^2\\
&=& \bigg|\bigg\langle \left(\begin{array}{cc}
	X&O \\
	O&W
\end{array}\right)x, x \bigg\rangle_B + \bigg\langle \left(\begin{array}{cc}
	O&Y \\
	Z&O
\end{array}\right)x, x \bigg\rangle_B  \bigg|^2 \\
&\leq& \bigg|\bigg\langle \left(\begin{array}{cc}
	X&O \\
	O&W
\end{array}\right)x, x \bigg\rangle_B\bigg|^2 +  \bigg|\bigg\langle \left(\begin{array}{cc}
	O&Y \\
	Z&O
\end{array}\right)x, x \bigg\rangle_B\bigg|^2 \\
&&+ 2  \bigg|\bigg\langle \left(\begin{array}{cc}
	X&O \\
	O&W
\end{array}\right)x, x \bigg\rangle_B\bigg\langle \left(\begin{array}{cc}
	O&Y \\
	Z&O
\end{array}\right)x, x \bigg\rangle_B\bigg|\\
&=& \bigg|\bigg\langle \left(\begin{array}{cc}
	X&O \\
	O&W
\end{array}\right)x, x \bigg\rangle_B\bigg|^2 +  \bigg|\bigg\langle \left(\begin{array}{cc}
	O&Y \\
	Z&O
\end{array}\right)x, x \bigg\rangle_B\bigg|^2 \\
&& +  2  \bigg|\bigg\langle \left(\begin{array}{cc}
	X&O \\
	O&W
\end{array}\right)x, x \bigg\rangle_B\bigg\langle x, \left(\begin{array}{cc}
O&Y \\
Z&O
\end{array}\right)x \bigg\rangle_B\bigg|\\
&\leq&  \bigg|\bigg\langle \left(\begin{array}{cc}
	X&O \\
	O&W
\end{array}\right)x, x \bigg\rangle_B\bigg|^2 +  \bigg|\bigg\langle \left(\begin{array}{cc}
	O&Y \\
	Z&O
\end{array}\right)x, x \bigg\rangle_B\bigg|^2 \\
&& + \bigg\|\left(\begin{array}{cc}
	X&O \\
	O&W
\end{array}\right)x\bigg\|_B \bigg\|\left(\begin{array}{cc}
O&Y \\
Z&O
\end{array}\right)x \bigg\|_B + \bigg|\bigg\langle \left(\begin{array}{cc}
X&O \\
O&W
\end{array}\right)x,\left(\begin{array}{cc}
O&Y \\
Z&O
\end{array}\right)x  \bigg\rangle_B\bigg|,\\
&& ~~\,\,\,\,\,\,\,\,\,\,\,\,\,\,\,\,\,\,\,\,\,\,\,\,\,\,\,\,\,\,\,\,\,\,\,\,\,\,\,\,\,\,\,\,\,\,\,\,\,\,\,\,\,\,\,\,\,\,\,\,\,\,\,\,\,\,\,\,\,\,\,\,\,\,\,\,\,\,\,\,\,\,\,\,\,\,\,\,\,\,\,\,\,\,\,\,\,\,\,\,\,\,\,\,\,\,\,\,\,\,\,\,\,\,\,\,\,\,\,\,\,\,\,\,\,\,\mbox{by Lemma \ref{Lemma:3}}\\
&=&   \bigg|\bigg\langle \left(\begin{array}{cc}
	X&O \\
	O&W
\end{array}\right)x, x \bigg\rangle_B\bigg|^2 +  \bigg|\bigg\langle \left(\begin{array}{cc}
	O&Y \\
	Z&O
\end{array}\right)x, x \bigg\rangle_B\bigg|^2 \\
&& +  \bigg\langle \left(\begin{array}{cc}
	X^{\sharp_A}X&O \\
	O&W^{\sharp_A}W
\end{array}\right)x, x  \bigg\rangle^{\frac{1}{2}}_B  \bigg\langle \left(\begin{array}{cc}
Z^{\sharp_A}Z&O \\
O&Y^{\sharp_A}Y
\end{array}\right)x, x  \bigg\rangle^{\frac{1}{2}}_B \\
&& +\bigg| \bigg\langle \left(\begin{array}{cc}
O&	Z^{\sharp_A}W \\
Y^{\sharp_A}X & O
\end{array}\right)x, x  \bigg \rangle_B\bigg|\\
&\leq&   \bigg|\bigg\langle \left(\begin{array}{cc}
	X&O \\
	O&W
\end{array}\right)x, x \bigg\rangle_B\bigg|^2 +  \bigg|\bigg\langle \left(\begin{array}{cc}
	O&Y \\
	Z&O
\end{array}\right)x, x \bigg\rangle_B\bigg|^2 \\
&&  +  \frac{1}{2}\bigg\langle \left(\begin{array}{cc}
	X^{\sharp_A}X + 	Z^{\sharp_A}Z&O \\
	O& Y^{\sharp_A}Y + W^{\sharp_A}W
\end{array}\right)x, x  \bigg\rangle_B \\
 && + \bigg|\bigg\langle \left(\begin{array}{cc}
	O&	Z^{\sharp_A}W \\
	Y^{\sharp_A}X & O
\end{array}\right)x, x  \bigg \rangle_B\bigg|,~~~ ~~~~\mbox{by AM-GM inequality}.
\end{eqnarray*}
Taking supremum over $\|x\|_B = 1$, and using Lemma \ref{Lemma:7}, we get the required inequality of the theorem.
\end{proof}

In particular,  considering $W=X,$ $ Z=Y$ in Theorem \ref{Theorem:14} and then using Lemma \ref{Lemma:7}, we get the following corollary.

\begin{cor}
Let $X, Y \in \mathcal{B}_A(\mathcal{H}).$ Then
	\[ \max \{w^2_A(X+Y), w^2_A(X-Y)\} \leq w_A^2(X) + w_A^2(Y) + \frac{1}{2}\|	X^{\sharp_A}X + 	Y^{\sharp_A}Y\|_A + w_A(Y^{\sharp_A}X).\]
\end{cor}

\begin{remark}
Consider  $T = \left(\begin{array}{cc}
	X&Y \\
	Z&W
	\end{array}\right),$ where  $Y= Z = \left(\begin{array}{cc}
	0&1\\
	0&0
	\end{array}\right),$ $ X = W =\left(\begin{array}{cc}
	0&0\\
	0&0
	\end{array}\right).$ If we take $  A = I$ then the bound obtained by Rout et. al. in \cite[Th. 3.5]{Rout} gives $w^2_B(T) \leq 4$, whereas the bound obtained in Theorem \ref{Theorem:14} gives $w^2_B(T) \leq \frac{3}{4}.$ Therefore, for this operator matrix T, the bound obtained in Theorem \ref{Theorem:14} is better than that in \cite[Th. 3.5]{Rout}.
\end{remark}

Finnaly, we obtain another upper bound for the $B$-numerical radius of  $2\times 2$ operator matrices of the form $\left(\begin{array}{cc}
X&Y \\
Z&W
\end{array}\right)$, where $X, Y,Z,W \in \mathcal{B}_A(\mathcal{H})$.

\begin{theorem} \label{Theorem:15}
	Let $X, Y, Z, W \in \mathcal{B}_A(\mathcal{H})$. Then
\begin{eqnarray*}
	w_B^2\left(\begin{array}{cc}
		X&Y \\
		Z&W
	\end{array}\right) &\leq & \max\{w_A^2(X), w_A^2(W)\}+ \frac{1}{2} \max\{w_A(YZ), w_A(ZY)\}\\
	&& +w_B  \left(\begin{array}{cc}
	O&YW \\
	ZX&O
\end{array}\right)  \\
	&& + \frac{1}{4}\max \{\|YY^{\sharp_A} + Z^{\sharp_A}Z\|_A, \| Y^{\sharp_A}Y + ZZ^{\sharp_A} \|_A \} \\
&& + \frac{1}{2} \max\{\| X^{\sharp_A}X + YY^{\sharp_A}\|_A, \| W^{\sharp_A}W+ZZ^{\sharp_A}\|_A\}.
\end{eqnarray*} 	
\end{theorem}

\begin{proof}
Let $x$ be a B-unit vector in $\mathcal{H} \oplus \mathcal{H}.$ Then 
\begin{eqnarray*}
	 \bigg|\bigg\langle \left(\begin{array}{cc}
	X&Y \\
	Z&W
\end{array}\right)x, x \bigg\rangle_B\bigg|^2
 &=& \bigg|\bigg\langle \bigg[\left(\begin{array}{cc}
	X&O \\
	O&W
\end{array}\right) + \left(\begin{array}{cc}
	O&Y \\
	Z&O
\end{array}\right)\bigg] x, x \bigg\rangle_B\bigg|^2\\
&=& \bigg|\bigg\langle \left(\begin{array}{cc}
	X&O \\
	O&W
\end{array}\right)x, x \bigg\rangle_B + \bigg\langle \left(\begin{array}{cc}
	O&Y \\
	Z&O
\end{array}\right)x, x \bigg\rangle_B  \bigg|^2 \\
&\leq& \bigg|\bigg\langle \left(\begin{array}{cc}
	X&O \\
	O&W
\end{array}\right)x, x \bigg\rangle_B\bigg|^2 +  \bigg|\bigg\langle \left(\begin{array}{cc}
	O&Y \\
	Z&O
\end{array}\right)x, x \bigg\rangle_B\bigg|^2 \\
&&+ 2  \bigg|\bigg\langle \left(\begin{array}{cc}
	X&O \\
	O&W
\end{array}\right)x, x \bigg\rangle_B\bigg\langle \left(\begin{array}{cc}
	O&Y \\
	Z&O
\end{array}\right)x, x \bigg\rangle_B\bigg|\\
&=&  \bigg|\bigg\langle \left(\begin{array}{cc}
	X&O \\
	O&W
\end{array}\right)x, x \bigg\rangle_B\bigg|^2 \\
&& + \bigg|\bigg\langle \left(\begin{array}{cc}
	O&Y \\
	Z&O
\end{array}\right)x, x \bigg\rangle_B \bigg\langle x,  \left(\begin{array}{cc}
O&Z^{\sharp_A} \\
Y^{\sharp_A}&O
\end{array}\right)x \bigg\rangle_B\bigg|\\
&& + 2 \bigg|\bigg\langle \left(\begin{array}{cc}
	X&O \\
	O&W
\end{array}\right)x, x \bigg\rangle_B \bigg\langle x,   \left(\begin{array}{cc}
O&Z^{\sharp_A} \\
Y^{\sharp_A}&O
\end{array}\right)x\bigg \rangle_B\bigg|
\end{eqnarray*}

\begin{eqnarray*}
&\leq& \bigg|\bigg\langle \left(\begin{array}{cc}
	X&O \\
	O&W
\end{array}\right)x, x \bigg\rangle_B\bigg|^2 \\
&& + \frac{1}{2} \bigg[\bigg\|\left(\begin{array}{cc}
	O&Y \\
	Z&O
\end{array}\right)x\bigg\|_B \bigg\| \left(\begin{array}{cc}
O&Z^{\sharp_A} \\
Y^{\sharp_A}&O
\end{array}\right)x\bigg\|_B + \bigg|\bigg\langle  \left(\begin{array}{cc}
O&Y \\
Z&O
\end{array}\right)x, \left(\begin{array}{cc}
O&Z^{\sharp_A} \\
Y^{\sharp_A}&O
\end{array}\right)x\bigg\rangle_B\bigg|\bigg] \\
&& + \bigg\|\left(\begin{array}{cc}
	X&O \\
	O&W
\end{array}\right)x\bigg\|_B \bigg\| \left(\begin{array}{cc}
	O&Z^{\sharp_A} \\
	Y^{\sharp_A}&O
\end{array}\right)x\bigg\|_B +\bigg| \bigg\langle  \left(\begin{array}{cc}
X&O \\
O&W
\end{array}\right)x, \left(\begin{array}{cc}
	O&Z^{\sharp_A} \\
	Y^{\sharp_A}&O
\end{array}\right)x\bigg\rangle_B \bigg| ,\\
&&~~\,\,\,\,\,\,\,\,\,\,\,\,\,\,\,\,\,\,\,\,\,\,\,\,\,\,\,\,\,\,\,\,\,\,\,\,\,\,\,\,\,\,\,\,\,\,\,\,\,\,\,\,\,\,\,\,\,\,\,\,\,\,\,\,\,\,\,\,\,\,\,\,\,\,\,\,\,\,\,\,\,\,\,\,\,\,\,\,\,\,\,\,\,\,\,\,\,\,\,\,\,\,\,\,\,\,\,\,\,\,\,\,\,\,\,\,\,\,\,\,\,\,\,\,\,\,\mbox{by Lemma \ref{Lemma:3}}\\
&=& \bigg|\bigg\langle \left(\begin{array}{cc}
	X&O \\
	O&W
\end{array}\right)x, x \bigg\rangle_B\bigg|^2 \\
&& + \frac{1}{2} \bigg\langle \left(\begin{array}{cc}
	Z^{\sharp_A}Z&O \\
	O&Y^{\sharp_A}Y
\end{array}\right)x, x \bigg\rangle^{\frac{1}{2}}_B  \bigg\langle \left(\begin{array}{cc}
YY^{\sharp_A}&O \\
O&ZZ^{\sharp_A}
\end{array}\right)x, x \bigg\rangle^{\frac{1}{2}}_B + \frac{1}{2}  \bigg|\bigg\langle \left(\begin{array}{cc}
	YZ&O \\
	O&ZY
\end{array}\right)x, x \bigg\rangle_B \bigg| \\
&& +  \bigg\langle \left(\begin{array}{cc}
	X^{\sharp_A}X&O \\
	O&W^{\sharp_A}W
\end{array}\right)x, x \bigg\rangle_B^{\frac{1}{2}}  \bigg\langle \left(\begin{array}{cc}
YY^{\sharp_A}&O \\
O&ZZ^{\sharp_A}
\end{array}\right)x, x \bigg\rangle^{\frac{1}{2}}_B  +  \bigg|\bigg\langle \left(\begin{array}{cc}
	O&YW \\
ZX&O
\end{array}\right)x, x \bigg\rangle_B\bigg|\\
\end{eqnarray*}
\begin{eqnarray*}
&\leq& \bigg|\bigg\langle \left(\begin{array}{cc}
	X&O \\
	O&W
\end{array}\right)x, x \bigg\rangle_B\bigg|^2 \\
&& + \frac{1}{4}  \bigg\langle \left(\begin{array}{cc}
	Z^{\sharp_A}Z+ YY^{\sharp_A}&O \\
	O&Y^{\sharp_A}Y + ZZ^{\sharp_A}
\end{array}\right)x, x \bigg\rangle_B +  \frac{1}{2} \bigg| \bigg\langle \left(\begin{array}{cc}
	YZ&O \\
	O&ZY
\end{array}\right)x, x \bigg\rangle_B \bigg|\\
&& + \frac{1}{2}  \bigg\langle \left(\begin{array}{cc}
	X^{\sharp_A}X + YY^{\sharp_A}&O \\
	O& W^{\sharp_A}W + ZZ^{\sharp_A}
\end{array}\right)x, x \bigg\rangle_B + \bigg| \bigg\langle \left(\begin{array}{cc}
	O&YW \\
	ZX&O
\end{array}\right)x, x \bigg\rangle_B\bigg|,\\
&&~~\,\,\,\,\,\,\,\,\,\,\,\,\,\,\,\,\,\,\,\,\,\,\,\,\,\,\,\,\,\,\,\,\,\,\,\,\,\,\,\,\,\,\,\,\,\,\,\,\,\,\,\,\,\,\,\,\,\,\,\,\,\,\,\,\,\,\,\,\,\,\,\,\,\,\,\,\,\,\,\,\,\,\,\,\,\,\,\,\,\,\,\,\,\,\,\,\,\,\,\,\,\,\,\,\,\,\,\,\,\,\,\,\,\,\,\,\,\,\,\,\,\,\,\,\,\,\mbox{by AM-GM inequality}.
\end{eqnarray*}
Taking supremum over $\|x\|_B = 1$ and then using Lemma \ref{Lemma:7}, we get the desired result.
\end{proof}

Considering $ W=X,$ $ Z=Y$ in Theorem \ref{Theorem:15} and using Lemma \ref{Lemma:7}, we get the following corollary.

\begin{cor}
	Let $X, Y \in \mathcal{B}_A(\mathcal{H}).$ Then
	\begin{eqnarray*}
 \max \big\{w^2_A(X+Y), w^2_A(X-Y)\big\} &\leq& w_A^2(X) + \frac{1}{4}\big\|YY^{\sharp_A} + Y^{\sharp_A}Y\big\|_A\\
	&& + \frac{1}{2} w_A(Y^2) + \frac{1}{2} \big\| X^{\sharp_A}X + YY^{\sharp_A}\big\|_A + w_A(YX).
	\end{eqnarray*}
\end{cor}

\begin{remark}
Consider $T = \left(\begin{array}{cc}
	X&Y \\
	Z&W
	\end{array}\right)$, where  $X =Z = W = O, Y= \left(\begin{array}{cc}
	0&1 \\
	0&0
	\end{array}\right)$ and $ A= I,$ then the bound obtained  in \cite[Th. 3.7]{Rout} gives $w^2_B(T) \leq 1$,  whereas the bound in Theorem \ref{Theorem:15} gives $w^2_B(T) \leq \frac{3}{4}$. Therefore, for this operator matrix T, Theorem \ref{Theorem:15} gives a better bound than that in \cite[Th. 3.7]{Rout}.
\end{remark}

\noindent \textbf{Incomparability of Theorem \ref{Theorem:14} and Theorem \ref{Theorem:15}.}
If we consider $X=W=O$ and $Y=Z$ then it follows from (\ref{1b}) that Theorem \ref{Theorem:14} gives better bound  than that in Theorem \ref{Theorem:15}.
Again, if we consider $A=I$ and $ X =(\frac{1}{2}), Y =(1), Z=W=(0) $ then for the operator matrix  $T = \left(\begin{array} {cc}
	X&Y \\
	Z&W
\end{array}\right)$  Theorem \ref{Theorem:14} gives $w^2_B(T)\leq \frac{10}{8}$, whereas Theorem \ref{Theorem:15} gives $w^2_B(T)\leq \frac{9}{8}.$ Therefore, for this operator matrix, Theorem \ref{Theorem:15} gives better bound  than that in Theorem \ref{Theorem:14}.
Thus, bounds obtained in Theorem \ref{Theorem:14} and Theorem \ref{Theorem:15} are not comparable, in general.\\

\bibliographystyle{amsplain}

\begin{thebibliography}{99}
	

\bibitem{Arias1} M.L. Arias, G. Corach and   M.C. Gonzalez, Partial isometries in semi-Hilbertian spaces, Linear Algebra Appl. 428 (2008) 1460-1475.	


\bibitem{Kittaneh:3} W. Bani-Domi and   F. Kittaneh, Norm and numerical radius inequalities for Hilbert space operators, Linear Multilinear Algebra (2020),
\url{https://doi.org/10.1080/03081087.2020.1798334}.



\bibitem{Rajejla} P. Bhunia, K. Paul and   R.K. Nayak, On inequalities for A-numerical radius of operators, Electron. J. Linear Algebra 36 (2020) 143-157.


\bibitem{RajAOT} P. Bhunia, R.K. Nayak and   K. Paul,  Refinements of A-numerical radius inequalities and their applications, Adv. Oper. Theory 5 (2020) 1498-1511.


\bibitem{Pintu1} P. Bhunia and   K. Paul,  Some  improvement of numerical radius inequalities of operators and operator matrices, Linear Multilinear Algebra (2020) \url{https://doi.org/10.1080/03081087.2020.1781037}.



\bibitem{Pintu2} P. Bhunia, K. Feki and   K. Paul,  A-Numerical radius orthogonality and parallelism of semi-Hilbertian space operators and their applications, Bull. Iran. Math. Soc. (2020) \url{https://doi.org/10.1007/s41980-020-00392-8}. 

\bibitem{Aniket} A. Bhanja, P. Bhunia and   K. Paul, On generalized Davis-Wielandt radius inequalities of semi-Hilbertian space operators, arXiv:2006.05069v1[math.FA].

\bibitem{BP} P. Bhunia and   K. Paul, Proper improvement of well-known numerical radius inequalities and their applications, \url{arXiv:2009.03206v1 [math.FA] }.

\bibitem{Douglas}  R.G. Douglas, On majorization, factorization and range inclusion of operators in Hilbert space, Proc. Amer. Math. Soc. 17 (1966) 413-416.


	
\bibitem{Kias} 	K. Feki, A note on the A-numerical radius of operators in semi-Hilbert spaces, Arch. Math. (2020) \url{https://doi.org/10.1007/s00013-020-01482-z}.



\bibitem{Kias2} K. Feki, Some A-numerical radius inequalities for $d \times d$ operator matrices,  arXiv:2003.14378 [math.FA]. 




\bibitem{Moslehian} M.S. Moslehian, Q. Xu and   A. Zamani, Seminorm and numerical radius inequalities of operators in semi-Hilbertian spaces, Linear Algebra Appl. 591 (2020) 299-321.	


\bibitem{SMY} M. Sattari, M.S. Moslehian and   T. Yamazaki, Some genaralized numerical radius inequalities for Hilbert space operators, Linear Algebra Appl. 470 (2015) 216-227.


\bibitem{Rout} N.C. Rout, S. Sahoo and  D. Mishra, On A-numerical radius inequalities for $ 2 \times 2 $  operator matrices, arxiv:2004.07494v1[math.FA].


\bibitem{Xu} Q. Xu, Z. Ye and A. Zamani, Some upper bounds for the A-numerical radius of  $ 2 \times 2 $  block matrices, Adv. Oper. Theory (2020), in press.


 \bibitem{Zamani} A. Zamani, A-numerical radius inequalities for semi-Hilbertian space operators, Linear Algebra Appl. 578 (2019) 159-183.	


\end{thebibliography}

\end{document}